\pdfoutput=1
\documentclass[11pt,a4paper]{article}
\usepackage[T1]{fontenc}
\usepackage[utf8]{inputenc}
\usepackage[english]{babel}
\usepackage[tbtags]{amsmath}
\usepackage{amsfonts,amssymb,amsthm}
\usepackage{lmodern}
\usepackage[a4paper,margin=25mm]{geometry}
\usepackage[hidelinks,hyperfootnotes=false]{hyperref}
\theoremstyle{plain}
\newtheorem{lemma}{Lemma}
\newtheorem{theorem}{Theorem}

\newcommand{\OmegaC}{\Omega^{\rm{cont}, 1}_{A((t))/A}}
\newcommand{\res}{\mathrm{Res}}
\newcommand{\Z}{\mathbb{Z}}

\title{Pairings on the algebra of Laurent series over a ring}
\author{Vladislav Levashev\\[0.5em]
  \small National Research University Higher School of Economics\\
  \small\texttt{valevashev@hse.ru}}
\date{}

\begin{document}
\maketitle
\begin{abstract}
	We prove that continuous $A$-bilinear pairings on the ring of Laurent
	series that are invariant under continuous automorphisms coincide,
	up to a constant, with the pairing given by the residue of a
	differential form over any commutative associative ring with identity.
\end{abstract}

\begingroup
\renewcommand{\thefootnote}{}
\footnotetext{The study has been funded within the framework of the HSE University Basic Research Program (HSE-BR-2025-060).
}
\endgroup
\medskip

Let $A$ be a commutative associative ring with identity, and let
$A((t)) = A[[t]][t^{-1}]$ be the ring of Laurent series over $A$.
The algebra $A((t))$ is equipped with the natural topology in which $A$-submodules
of the form $t^{m}A[[t]]$ form a  neighbourhood basis at zero.
Also, let $\Omega^{1}_{A((t))/A}$ be the module of Kähler
1-differentials of $A((t))$ over $A$.
Define the module of continuous Kähler 1-differentials as follows:

$$\OmegaC = \Omega^{1}_{A((t))/A}/Q,$$
where $Q$ is generated by elements of the form
$df - \frac{\partial f}{\partial t}dt$, $f\in A((t))$.
By definition, every element
$\omega\in \OmegaC$ can be written uniquely
as $(\sum a_{i} t^{i}) dt$.
Put $\res(\omega) = a_{-1}$. We also define
$$\res(f, g) = \res(f\,dg),$$
where $f\,dg \in \OmegaC$. This defines
an $A$-bilinear pairing on $A((t))$, which is continuous in each argument, i.e., continuous in either argument when the other is fixed. Moreover, if $\varphi$ is a continuous $A$-algebra
automorphism of $A((t))$, then
$\res(f,g) = \res(\varphi(f), \varphi(g))$
(see \cite[Ch.~II, §~11]{Ser68},
\cite[Proposition 5.3]{GorOsi16}).

For $l \geq 2$, denote by $\varphi_l$ the continuous $A$-linear
automorphism of the algebra $A((t))$ under which $t$ is sent
to $t + t^l$. This uniquely determines it by the rule:
$$\varphi_l(\sum_{k} a_{k} t^{k}) = \sum_{k} a_k (t + t^l)^{k}.$$

We prove the following result, generalizing Theorem 4
of \cite{Lev26} (see also Lemma 5.7 and Lemma 5.8
of \cite{OsiZhu16}):

\begin{theorem}
	Let $(\cdot, \cdot) \colon A((t)) \times A((t)) \to A$ be an $A$-bilinear
	pairing that is continuous in each argument.
	Suppose that for all $f,g \in A((t))$ and $l \geq 2$ we have
	$(f,g) = (\varphi_l(f), \varphi_l(g))$.
	Then there exists an element $e\in A$ such that
	$(f,g) = e \res(f, g)$.
\end{theorem}
\begin{proof}
	For any elements $a, b \in \Z$, denote by $R_{a,b}$
	the set of pairs $(x,y) \in \Z^2$ such that $x \geq a, y \geq b$. We will divide the proof into several steps.

		 1) First, we prove that there exist elements $a,b \in \Z$
		such that $(t^n, t^m) = 0$ for $(n, m) \in R_{a,b}$.
		Set $a = 2$. By assumption, the map
		$$f\in A((t)) \mapsto (t,f)\in A$$ is continuous.
		Consequently, there exists $b \in \Z$ such that
		$(t, g) = 0$ for $g \in t^{b}A[[t]]$.
		We claim that these choices of a and b have the required property.
		For an arbitrary $l \geq 2$, by assumption we have
		$$ 0 = (t, g) = (\varphi_{l}(t), \varphi_{l}(g)) = (t^l, \varphi_{l}(g)).$$
		Since g was arbitrary, the claim follows from an equality $\varphi_l(t^{b}A[[t]]) = t^{b}A[[t]]$.
		
		 2) Now, we prove the following implication:
		if $(t^n, t^m) = 0$ for $(n,m) \in R_{a, b+1} \cup R_{a+1, b}$,
		then $(t^n, t^m) = 0$ for $(n,m) \in R_{a-1, b+2} \cup R_{a+2, b-1}$.
We will prove this property for pairs $(n,m) \in R_{a+2, b-1}$,
since for the remaining pairs the argument is similar.

		We need to prove that $(t^{a'}, t^{b-1}) = 0$ for $a' > a+1$.
		If $(t^{a'}, t^{b-1}) = 0$ for all $a'$, there is nothing to prove.
		Otherwise, by continuity, there exists $x \in \Z$
		such that $(t^{x}, t^{b-1}) \neq 0$ and for all $x' > x$
		we have $(t^{x'}, t^{b-1}) = 0$.
		If $x \leq a+1$, we are done.
		Suppose that $x \geq a + 2$. Then by assumption:
		
		$$(t^{x-1}, t^{b-1}) = (\varphi_2(t)^{x-1}, \varphi_2(t)^{b-1}) = (t^{x-1}, t^{b-1}) + (x-1)(t^{x}, t^{b-1}), $$
		since the remaining summands either contain factors of the form
		$(t^{x'}, t^{b-1})$, where $x' \geq x+1$, hence
		are zero by the choice of $x$, or factors $(t^{x'}, t^{y'})$,
		where $x' \geq x-1 \geq a + 1$ and $y' \geq b$, hence
		are zero by assumption.
		Therefore $(x-1)(t^{x}, t^{b-1}) = 0$.
		Similarly, by assumption we have the equality
		$$ (t^{x-2}, t^{b-1}) = (\varphi_3(t)^{x-2}, \varphi_3(t)^{b-1}) = (t^{x-2}, t^{b-1}) + (x-2)(t^{x}, t^{b-1}), $$
		since the remaining summands either contain factors
		$(t^{x'}, t^{b-1})$, where $x' \geq x+1$, or factors
		$(t^{x'}, t^{y'})$, where $x' \geq x-2 \geq a$ and $y' \geq b+1$.
		From this we obtain $(x-2)(t^{x}, t^{b-1}) = 0$.
		Therefore $(t^{x}, t^{b-1}) = 0$.
		This contradicts the choice of $x$.

		3) For $k \in \Z$, denote by $D_{k}$ the set of pairs of integers
		$(x,y)\in \Z^2$ such that $x + y = k$.
		We claim that there exists $k$ such that for all $k' \geq k$
		and $(n,m) \in D_{k'}$ we have $(t^n, t^m) = 0$.
		Indeed, let $a,b \in \Z$ be as in step 1.
		Then put $k = a + b + 1$.
		Note that
		$$\bigsqcup_{k' \geq k} D_{k'} = \bigcup\limits_{a'+b' = k} R_{a',b'}.$$
		
		Moreover, $R_{a,b+1} \cup R_{a+1,b} \subset R_{a,b}.$
		Then it follows from step 2 that $(t^n, t^m) = 0$
		for $(n,m) \in R_{a-1, b+2} \cup R_{a+2,b-1}$. 
		Continuing similarly, we obtain the claim by induction.

		4)Let $k$ be such that $(t^n, t^m) = 0$ for $n+m \geq k$.
		In this step, we prove that for any $k' < k$ and any
		$(a,b), (a',b') \in D_{k'}$ with
		$a' - a = b - b' \geq k - k'$, the following identity holds:
		\begin{equation}\label{maineq}
			a(t^{a'}, t^{b'}) + b' (t^{a}, t^{b}) = 0.
		\end{equation}
		Indeed, set $\delta = a' - a = b - b' \geq k - k'.$
		By assumption, the following equality holds:
		$$(t^a, t^{b'})  = (\varphi_{\delta + 1}(t)^a, \varphi_{\delta + 1}(t)^{b'}) = (t^a, t^{b'}) + b'(t^a, t^{b}) + a(t^{a'}, t^{b'}),$$
		since all the remaining summands have a factor
		$(t^{x}, t^{y}),$ where either $x \geq a + 2\delta = a' + \delta$
		and $y\geq b'$, or $y \geq b' + 2\delta = b + \delta$
		and $x \geq a$, or $x = a + \delta, y = b'+\delta$. In each case $x + y \geq k' + \delta \geq k$.
		From this we obtain the required identity.
		In particular, if $b' = 1$, $a' = a + b - 1$
		and $a + 2b \geq k + 1$, we obtain
		$(t^a, t^b) = -a(t^{a + b - 1}, t)$.

		5)Fix $a, b \in \Z$. Let us prove that
		$(t^a, t^b) = b (t^{a+b-1}, t)$. Let $k \in \Z$ be such that $(t^n, t^m) = 0$ for $n + m \geq k$.
		If $a + b \geq k$, this holds since both sides are zero.
		
		Otherwise, put $k' = a + b$ and choose a pair $(c,d)$ such that
		the following three conditions hold:
		\begin{align}\label{cond}
		k' = c + d, \  a-c = d-b \geq k-k', \ c + 2d \geq k + 1
		\end{align}
		To obtain such a pair, choose an integer
		$c \in \Z$ satisfying
		$c \leq 2a + b - k$, $c \leq 2a + 2b - k - 1$,
		and put $d = k' - c$.
		Then by step 4, the third condition in (\ref{cond}) yields the equality
		$$(t^{c}, t^{d}) = -c(t^{a + b - 1}, t).$$
		
		On the other hand, apply identity (\ref{maineq})
		to the pairs $(c,d)$ and $(a,b)$:
		$$c(t^a, t^b) + b (t^c, t^d) = c((t^a, t^b) - b (t^{a+b-1}, t)) = 0.$$
		
		Since $c-1$ could also have been chosen in place of $c$, the claim follows.

		 6) Define an element $h \in A((t))$ by the equality
		$h = \sum_{k \in \Z} (t^{-k - 1}, t) t^{k}$.
		We claim that for any $f, g \in A((t))$ the following equality holds:
		\begin{align}\label{resh}
		(f,g) = \res(hf dg).
		\end{align}
		Indeed, by continuity in each argument it suffices to check equality (\ref{resh})
		for $f = t^a, g = t^b$.
		Then the left-hand side of the equality is $(t^a, t^b)$,
		and the right-hand side is $b(t^{a + b - 1}, t)$.
		This equality was verified in step 5.
		
		7) Now note that since $(\cdot, \cdot)$ is invariant
		under the automorphisms $\varphi_{l}$,
		for all $l \geq 2$ we have
		$$\res(h fdg) = \res(h \varphi_l(f) d \varphi_l(g)).$$
		
		On the other hand, since $\res$ is also invariant
		under $\varphi_{l}$,
		the right-hand side equals \linebreak{$\res(\varphi_{l}^{-1}(h) f dg).$}
		Setting $g = t$, $f = t^s$, we obtain
		that $h = \varphi_l(h)$ for any $l \geq 2$.
		The rest follows from Lemma~\ref{Hzero} and Lemma~\ref{fixed} given below.
\end{proof}
\begin{lemma}\label{Hzero}
Let $H_{i}\in A((t))$, $i\geq 1$, be a collection of elements for which there exists $m \in \Z$ such that $H_{i}\in t^{m-i}A[[t]]$. Suppose that for every $n\geq 2$ the following holds:
\begin{align}\label{sumHzero}
\sum_{i\geq 1} H_{i}t^{in} = 0.
\end{align}
Then $H_{i}=0$ for all $i\geq 1$.
\end{lemma}
\begin{proof}
Assume the contrary. Denote by $d_{i}\in \Z$ the maximal number with the property $H_{i} \in t^{d_{i}}A[[t]]$ if $H_{i}\neq 0$, and otherwise $d_{i}= +\infty$. From this it follows that $H_{i}t^{in}\in t^{d_{i}+in}A[[t]]$. Note that $d_{i}\geq m-i$, hence for any fixed $n\geq 2$ the set of numbers $\{d_{i}+ni\}_{i\geq 1}$ always has a finite minimum (since $H_{i}\neq 0$ for some $i$).
  Denote by $j_{n}$ the smallest number with the property:
\begin{align}\label{ineq}
d_{j_n} + n j_{n} \leq d_{j} + nj, \quad j\geq 1.
\end{align}
Note that $j_{n} \geq j_{n+1}$. Indeed, from the definition of $j_{n+1}$ we have
$$ d_{j_{n+1}} + (n+1) j_{n+1} \leq d_{j_{n}} + (n+1)j_{n}.$$
Rearranging the terms, we obtain:
$$j_{n} - j_{n+1} \geq (d_{j_{n+1}} + nj_{n+1}) - (d_{j_{n}} + nj_{n}) \geq 0,$$
where the last inequality follows from the definition of $j_{n}$. Hence the sequence of numbers $\{j_{n}\}_{n\geq 2}$ stabilizes for $n\geq N$.
Set
$$j_0 = j_{N+k}, \qquad d_{0} = d_{j_{N+k}}, \quad k\geq 1.$$
Then we claim that for $j \neq j_0$: $d_{0} + (N+1)j_{0} < d_{j} + (N+1)j$, i.e., the inequality in (\ref{ineq}) for $n = N+1$ is strict. Indeed, if $j_{0} < j$, then by the hypothesis $d_{0} + N j_{0} \leq d_{j} + Nj$, therefore adding $j_{0}$ to the left-hand side and $j$ to the right-hand side we obtain the required inequality. If $j_{0} > j$ and
$ d_{0} + (N+1) j_{0} = d_{j} + (N+1)j,$
then a contradiction arises with the minimality of $j_{0} = j_{N+1}$.

Write $H_{j_{0}} = \sum_{k} a_{k}t^{k}$ and consider equality (\ref{sumHzero}) for $n = N+1$. The coefficient of $t^{d_{0}+(N+1)j_{0}}$ equals $a_{d_{0}}$, hence $a_{d_0} = 0$. We obtain a contradiction with the definition of $d_{0}$. Consequently, $H_{i}=0$ for every $i\geq 1$. The lemma is proved.
\end{proof}

For the proof of Lemma~2 we will need the Hasse–Schmidt derivatives, which are defined as follows: if $f = \sum_{k} a_{k}t^{k}\in A((t))$ and $j \geq  0$, then put
$$f^{[j]} = \sum_{k} \binom{k}{j} a_{k}t^{k-j}, \quad \text{where}\quad \binom{k}{j} = \frac{k(k-1)\cdots (k-j+1)}{j!},\ k\in \mathbb{Z}.$$
In particular, $f^{[0]} = f$, $f^{[1]} = f'$ and the element $j! f^{[j]}$ equals the ordinary $j$-th derivative of the element $f$. If $\varphi$ is a continuous $A$-linear automorphism of the algebra $A((t))$ such that $\varphi(t) = t + g$, where $g\in t^{2}A[[t]]$, then the following version of the Taylor formula holds:
\begin{equation}\label{taylor}
\varphi(f) = f(t+g) = \sum_{j\geq 0} f^{[j]}g^{j}.
\end{equation}

\begin{lemma}\label{fixed}
Let $h\in A((t))$ be a series such that $\varphi_{n}(h) = h$ for all $n\geq 2$. Then $h\in A$.
\end{lemma}
\begin{proof}
Let $h\in t^{m}A[[t]]$. From the condition and formula (\ref{taylor}) above, for all $n\geq 2$ we have:
\begin{equation}\label{sumeqzero}
\sum_{j\geq 1} h^{[j]}t^{jn} = 0,
\end{equation}
where $h^{[j]}\in t^{m-j}A[[t]]$. Then by Lemma~\ref{Hzero} we have $h^{[j]}=0$ for $j\geq 1$. Now put $h = \sum_{k} a_k t^{k}$. The preceding equalities imply that $\binom{k}{j}a_{k}=0$ for $k\neq 0$, $j\geq 1$. Next, note that for any fixed $k\neq 0$ the set of binomial coefficients $\{\binom{k}{j}\}_{j\geq 1}$ generates the unit ideal in $\Z$ (otherwise there exists a prime $p$ such that the elements $\binom{k}{j}$ vanish in the field of $p$ elements $\mathbb{F}_{p}$ for $j\geq 1$, hence the element $(1+x)^{k}\in \mathbb{F}_{p}[[x]]$ equals $1$, which is not the case). From this we obtain that $a_k = 0$ for every $k\neq 0$. The lemma is proved.
\end{proof}

In the paper \cite{Lev26} the author of the present note proved an analogous result for residues on the algebra of iterated Laurent series $A((t_1))\cdots((t_n))$ under some restrictions on the ring. For example, the result of \cite{Lev26} is inapplicable to the case $A = \mathbb{F}_{p}$. A natural and interesting question arises: can Theorem~1 be generalized to the case $n>1$ using the multidimensional Parshin residue (see \cite{Par84}, \cite{GorOsi16})?

The author expresses gratitude to Denis Vasilievich Osipov for valuable comments and interest in the work.


\begin{thebibliography}{99}
\small

\bibitem{GorOsi16}
S.~O.~Gorchinskiy, D.~V.~Osipov.
\newblock Continuous homomorphisms between algebras of iterated
Laurent series over a ring.
\newblock \emph{Proc. Steklov Inst. Math.},
\textbf{294} (2016), 47--66.
\newblock \url{https://doi.org/10.1134/S0081543816060031}.


\bibitem{Lev26}
V.~A.~Levashev.
\newblock Polymultiplicative maps associated with the algebra of iterated Laurent series and the higher-dimensional Contou--Carr\`ere symbol.
\newblock \emph{Mat. Sb.}, \textbf{217}, no.~4 (2026), 106--136.
\newblock \url{http://mi.mathnet.ru/sm10308}.
\newblock \url{https://doi.org/10.4213/sm10308}.

\bibitem{OsiZhu16}
D.~V.~Osipov, Xinwen~Zhu.
\newblock The two-dimensional Contou--Carr\`ere symbol and reciprocity laws.
\newblock \emph{J. Algebraic Geom.}, \textbf{25} (2016), 703--774.

\bibitem{Par84}
A.~N.~Parshin.
\newblock Local class field theory.
\newblock In \emph{Algebraic geometry and its applications},
Collected papers.
\newblock \emph{Trudy Mat. Inst. Steklov.},
\textbf{165} (1984), 143--170.
\newblock English translation:
\emph{Proc. Steklov Inst. Math.},
\textbf{165} (1985), 157--185.
\newblock \url{https://www.mathnet.ru/eng/tm2278}.


\bibitem{Ser68}
J.-P.~Serre.
\newblock \emph{Groupes alg\'ebriques et corps de classes}.
\newblock Publ. Inst. Math. Univ. Nancago, VII.
\newblock Hermann, Paris, 1959, 202 pp.

\end{thebibliography}
\end{document}